\documentclass[a4paper,11]{article}

\usepackage{amsmath,amssymb,amsthm}

\usepackage{float}

\def\eps{\varepsilon}
\def\be{\begin{equation}}
\def\ee{\end{equation}}
\def\dss{\displaystyle}

\newtheorem{lemma}{Lemma}

\newtheorem{thm}{Theorem}
\theoremstyle{definition}
\newtheorem{remark}{Remark}

\begin{document}

\title{Uniform convergence on a Bakhvalov-type mesh using the preconditioning approach:\\ Technical report}
\author{Th\'{a}i Anh Nhan$^{a,}$\thanks{This author's work is supported by the Irish Research
Council under Grant No. RS/2011/179.}~ and Relja Vulanovi\'c$^{b}$\\
$^{a}$\small{School of Mathematics, Statistics and Applied Mathematics,}\\
\small{National University of Ireland, Galway, Ireland}\\
$^{b}$\small{Department of Mathematical Sciences,
Kent State University at Stark,}\\
\small{6000 Frank Ave. NW, North Canton, OH 44720, USA}\\
\small{Email: a.nhan1@nuigalway.ie, 
rvulanov@kent.edu 
}}

\maketitle
\begin{abstract}
The linear singularly perturbed convection-diffusion problem in one
dimension is considered and its discretization on a Bakhvalov-type
mesh is analyzed. The preconditioning technique is used to obtain
the pointwise convergence uniform in the perturbation parameter.
\end{abstract}
\section{Introduction}\label{sec:intro}
The report is a supplement to~\cite{NV15}.
\section{The continuous problem}

We consider the problem
\begin{equation}\label{eq:1DCD}
\mathcal{L}u:= -\varepsilon u'' -b(x)u' + c(x)u=f(x),\ x\in(0,1),\
u(0)=u(1)=0,
\end{equation}
with a small positive perturbation parameter $\eps$ and
$C^1[0,1]$-functions $b$, $c$, and $f$, where $b$ and $c$ satisfy
\[
b(x)\ge \beta>0, \ \ c(x)\ge 0 \ \ \mbox{for $x\in I:=[0,1]$}.
\]
It is well known, see \cite{KT78, Lorenz} for instance, that
\eqref{eq:1DCD} has a unique solution $u$ in $C^3(I)$, which in
general has a boundary layer near $x=0$. Our goal is to find this
solution numerically.

The solution $u$ can be decomposed into the smooth and
boundary-layer parts. We present here Lin\ss's \cite[Theorem
3.48]{Linss10} version of such a decomposition:
\begin{equation}\label{eq:decomp}
 u(x) = s(x) + y(x),
\end{equation}
\begin{equation}\label{eq:decomp_est}
|s^{(k)}(x)|\le C\left(1+\eps^{2-k}\right), \quad |y^{(k)}(x)|\le
C\eps^{-k}e^{-\beta x/\eps},
\end{equation}
\[ x\in I, \quad k=0,1,2,3. \]
Above and throughout the report, $C$ denotes a generic positive
constant which is independent of $\eps$. For the construction of the
function $s$, see \cite{Linss10}, since the details are not of
interest here. As for $y$, it is important to note that it solves
the problem
\begin{equation*}\label{eq:homogeneous}
\mathcal{L}y(x)=0, \quad x\in (0,1),\quad y(0)=-s(0),\quad y(1)=0,
\end{equation*}
with a homogeneous differential equation. We shall use this fact
later on in the report.

\section{The discrete problem and condition number estimate}

We first define a finite-difference discretization of the problem
\eqref{eq:1DCD} on a general mesh $I^N$ with mesh points $x_i$,
$i=0,1,\ldots,N$, such that $0=x_0<x_1<\dots<x_N=1$. Throughout the
rest of the paper, the constants $C$ are also independent of $N$.

Let $h_i=x_i-x_{i-1}$, $i=1,2,\ldots, N$, and
$\hbar_i=(h_i+h_{i+1})/2$, $i=1,2,\ldots,N-1$. Mesh functions on
$I^N$ are denoted by $W^N$, $U^N$, etc. If $g$ is a function defined
on $I$, we write $g_i$ instead of $g(x_i)$ and $g^N$ for the
corresponding mesh function. Any mesh function $W^N$ is identified
with an $(N+1)$-dimensional column vector, $W^N=[W_0^N,
W_1^N,\ldots,W_N^N]^T$, and its maximum norm is given by
\[
\left\|W^N\right\| = \max_{0\le i\le N}|W^N_i| .
\]
For the matrix norm, which we also denote by $\|\cdot\|$, we take
the norm subordinate to the above maximum vector norm.

We discretize the problem \eqref{eq:1DCD} on $I^N$ using the upwind
finite-difference scheme:
\[ U^N_0  =  0, \]
\begin{equation}\label{eq:FD_scheme}
\mathcal{L}^NU_i^N:=  -\eps D'' U^N_i - b_iD'U^N_i + c_iU_i^N = f_i,
\quad i=1,2,\ldots, N-1,
\end{equation}
\[ U^N_N = 0, \]
where
\[
D''W^N_i=\frac{1}{\hbar_i}\left(\frac{W^N_{i+1}-W^N_i}{h_{i+1}}-
\frac{W^N_i-W^N_{i-1}}{h_i}\right)
\]
and
\[
D'W^N_i=\frac{W^N_{i+1}-W^N_i}{h_{i+1}}.
\]
The linear system \eqref{eq:FD_scheme} can be written down in matrix
form,
\begin{equation}\label{eq:matrix_form}
A_N U^N = \hat f^N,
\end{equation}
where $A_N=[a_{ij}]$ is a tridiagonal matrix with $a_{00}=1$ and
$a_{NN}=1$ being the only nonzero elements in the 0th and $N$th
rows, respectively, and where $\hat f^N = [0,f_1,f_2,\ldots,
f_{N-1}, 0]^T$.

It is easy to see that $A_N$ is an $L$-matrix, i.e., $a_{ii}>0$ and
$a_{ij}\le 0$ if $i\neq j$, for all $i,j =0,1,\ldots, N$. The matrix
$A_N$ is also inverse monotone, which means that it is non-singular
and that $A_N^{-1} \ge 0$ (inequalities involving matrices and
vectors should be understood component-wise), and therefore an
$M$-matrix (inverse monotone $L$-matrix). This can be proved using
the following $M$-criterion, see \cite{bo81} for instance.

\begin{thm}\label{M_crit}
Let $A$ be an $L$-matrix and let there exist a vector $w$ such that
$w>0$ and $Aw \ge \gamma$ for some positive constant $\gamma$. $A$
is then an $M$-matrix and it holds that $\| A^{-1}\| \le \gamma^{-1}
\| w \|$.
\end{thm}

To see that $A_N$ is an $M$-matrix, just set $w_i = 2-x_i$,
$i=0,1,\ldots, N$ in Theorem \ref{M_crit} to get that $A_Nw \ge
\min\{1, \beta\}$. This also implies that the discrete problem
\eqref{eq:matrix_form} is stable uniformly in $\eps$,
\begin{equation}\label{eq:A_stable}
\| A_N^{-1}\| \le \frac{2}{\min\{1, \beta\}} \le C .
\end{equation}
Of course, the system \eqref{eq:matrix_form} has a unique solution
$U^N$.

\section{A Bakhvalov-type mesh}

A generalization of the Bakhvalov mesh~\cite{Bakh} to a class of
Bakhvalov-type meshes can be found in~\cite{Vul83}. Here we take one
of the Bakhvalov-type meshes from~\cite{Vul83} for the
discretization mesh $I^N$. We refer to this mesh as as
Vulanovi\'c-Bakhvalov mesh (VB-mesh). The points of the VB-mesh are
generated by the function $\lambda$ in the sense that
$x_i=\lambda(t_i)$, where $t_i=i/N$. The mesh-generating function
$\lambda$ is defined as follows:
\begin{equation}\label{eq:mesh_generating_lambda}
 \lambda(t)=
\begin{cases}
  \psi(t), & t\in[0,\alpha],\\
   \psi(\alpha)+\psi'(\alpha)(t-\alpha), & t\in [\alpha,1],
\end{cases}
\end{equation}
with $0<q<1$ and $\psi = a\eps \phi$, where
\begin{equation*}
\phi(t)=\frac{t}{q-t}=\frac{q}{q-t}-1, \ \ t\in [0,\alpha].
\end{equation*}
On the interval $[\alpha, 1]$, $\lambda$ is the tangent line from
the point $(1,1)$ to $\psi$, touching $\psi$ at $(\alpha,
\psi(\alpha))$. The point $\alpha$ can be determined from the
equation
\[
\psi(\alpha) + \psi'(\alpha)(1-\alpha) = 1 .
\]
Since $\phi'(t) = q/(q-t)^2$, the above equation reduces to a
quadratic one,
\[
a\eps\alpha(q-\alpha) +a\eps q(1-\alpha) = (q-\alpha)^2 ,
\]
which is easy to solve for $\alpha$:
\[
\alpha = \frac{q - \sqrt{a\eps q(1-q+a\eps)}}{1+a\eps} .
\]
We have to assume that $a\eps <q$ (which is equivalent to
$\psi'(0)<1$) and then $\alpha > 0$. Note also that $\alpha < q$ and
\begin{equation}
\label{eq:zeta} q-\alpha = \zeta \sqrt{\eps}, \ \  \zeta \le C, \ \
\frac{1}{\zeta} \le C .
\end{equation}
Let $J$ be the index such that $t_{J-1} < \alpha \le t_J$. Starting
from the mesh point $x_J$, the mesh is uniform, with step size $H$.
However, $x_J$ behaves differently from the transition point of the
Shishkin mesh because
\[
x_J \ge \psi(\alpha) = \frac{a\alpha}{\zeta}\sqrt{\eps} .
\]
We note that the transition point $\psi(\alpha)$ is different also
from the Bakhvalov-Shishkin of Vulanovi\'c-Shishkin meshes in the
sense of~\cite{RL99}.

We now give the estimate for the condition number of $A_N$ when the
discrete problem \eqref{eq:FD_scheme} is formed on the VB-mesh as
described above. The condition number is
\[
\kappa(A_N) := \| A_N^{-1}\| \| A_N\|.
\]
We estimate the upper bound for $\|A_N\|$ by examining the entries
of the matrix $A_N$ directly,
\begin{equation*}\label{eq:est_A_N}
\|A_N\|\le C\frac{N^2}{\eps}.
\end{equation*}
Combining this with \eqref{eq:A_stable}, we get the following
result.

\begin{thm}\label{thm:cond_A}
The condition number of $A_N$ on the VB-mesh satisfies the following
sharp bound:
\begin{equation*}
\kappa(A_N)\le C\frac{N^2}{\eps}.
\end{equation*}
\end{thm}

\section{Conditioning}
Let $M=\mbox{diag }(m_0,m_1,\ldots, m_N)$ be a diagonal matrix with
the entries
\[
m_0 = 1, \ \ m_i = \frac{\hbar_{i}}{H}, \ i=1,2,\ldots, N-1, \ \
\mbox{and} \ \ m_N = 1.
\]
In other words,
\begin{equation}\label{eq:preconditioner}
m_0 = 1, \ \ m_i = \frac{\hbar_{i}}{H}, \ i=1,2,\ldots, J, \ \
\mbox{and} \ \ m_i = 1, i=J+1,\ldots, N.
\end{equation}
When the system \eqref{eq:matrix_form} is multiplied by $M$, this is
equivalent to multiplying the equations 1, 2, \ldots, $J$ of the
discrete problem \eqref{eq:FD_scheme}
 by $\hbar_{i}/H, i=1,2,\ldots, J$. The modified
system is
\begin{equation}\label{eq:mod_system}
\tilde A_N U^N = M\tilde f^N,
\end{equation}
where $\tilde A_N = MA_N$. Let the entries of $\tilde A_N$ be
denoted by $\tilde a_{ij}$, the nonzero ones being
\[
l_i := \tilde a_{i-1,i} = \left\{\begin{array}{ll}
\dss -\frac{\eps}{h_iH}, &  1\le i \le J-1, \\
 & \\
\dss -\frac{\eps}{h_JH}, & i=J, \\
 & \\
\dss - \frac{\eps}{H^2}, & J+1 \le i \le N-1,
\end{array}\right.
\]
\[
r_i := \tilde a_{i,i+1} = \left\{\begin{array}{ll}
\dss -\frac{\eps}{h_{i+1}H} - \frac{b_i\hbar_i}{h_{i+1}H}, &  1\le i \le J-1, \\
 & \\
\dss -\frac{\eps}{H^2} - \frac{b_i\hbar_i}{H^2}, & i=J, \\
 & \\
\dss - \frac{\eps}{H^2} - \frac{b_i}{H}, & J+1 \le i \le N-1,
\end{array}\right.
\]
and
\[
d_i := \tilde a_{ii} = \left\{\begin{array}{cl}
1, & i=0 \\
& \\
-l_i-r_i+ \dss \frac{\hbar_i}{H}c_i, &  1\le i \le J, \\
 & \\
-l_i-r_i+c_i, & J+1 \le i \le N-1, \\
 & \\
 1, & i=N.
\end{array}\right.
\]

Unlike the Shishkin mesh, which is piece-wise uniform, the VB-mesh
is graded in the fine part. Because of this, it is more difficulty
to prove the uniform stability of the modified scheme. This is done
in Lemma~\ref{lem:stability} below, but first we need some crucial
estimates for the graded mesh defined
by~\eqref{eq:mesh_generating_lambda}.

\begin{lemma}
For the mesh-generating function given
in~\eqref{eq:mesh_generating_lambda}, the following estimates hold
true:
\begin{equation}\label{eq:ratio_less_J-2}
\frac{\eps(h_{i+1}-h_i)}{h_ih_{i+1}}\le \frac{2}{a},\qquad i =
1,2,\ldots,J-2,
\end{equation}
and
\begin{equation}\label{eq:ratio_J}
\frac{\eps(H-h_J)}{h_JH}\le \frac{\zeta\sqrt{\eps}}{aq}.
\end{equation}
\end{lemma}

\begin{proof}
For $i\le J-2$, we have
\begin{equation*}\label{eq:hi}
\begin{split}
h_i=x_i-x_{i-1}&=a\eps\left( \frac{q}{q-t_i} - \frac{q}{q-t_{i-1}}\right)=\frac{a\eps q}{N(q-t_{i-1})(q-t_i)},\\
h_{i+1}&=\frac{a\eps q}{N(q-t_i)(q-t_{i+1})},
\end{split}
\end{equation*}
and
\[
h_{i+1}-h_i=\frac{2a\eps q}{N^2(q-t_{i-1})(q-t_i)(q-t_{i+1})}.
\]
Then \eqref{eq:ratio_less_J-2} follows because
\[
\frac{\eps(h_{i+1}-h_{i})}{h_ih_{i+1}}=\frac{2(q-t_i)}{aq}=\frac{2}{a}\left(1-\frac{t_i}{q}\right)\le
\frac{2}{a}.
\]

The proof of \eqref{eq:ratio_J} is more complicated due to the
presence of $h_J$. First, $h_J=\gamma_1 +\gamma_2$, where
$\gamma_1=x_\alpha-x_{J-1}$, $\gamma_2=x_J-x_\alpha$, and
$x_\alpha=\psi(\alpha)$. Since
\[
\begin{split}
\gamma_2&=\psi'(\alpha)(t_J-\alpha)\\
&=\frac{a\eps q}{q-\alpha}\left(\frac{t_J-\alpha}{q-\alpha}\right)\\
\end{split}
\]
and
\[
\begin{split}
\gamma_1&=a\eps\left(\phi(\alpha)-\phi\left(t_{J-1}\right)\right)\\
&=a\eps\left(\frac{\alpha}{q-\alpha}-\frac{t_{J-1}}{q-t_{J-1}}\right)\\
&=\frac{a\eps q}{q-\alpha}\cdot\frac{\alpha-t_{J-1}}{q-t_{J-1}},
\end{split}
\]
we have
\begin{equation*}\label{eq:h_J}
\begin{split}
h_J&=\frac{a\eps
q}{q-\alpha}\left[\frac{t_J - \alpha}{q-\alpha}+\frac{\alpha-t_{J-1}}{q-t_{J-1}}\right]\\
&=\frac{a \eps q}{(q-\alpha)^2}\left[t_J-\alpha+\frac{(q-\alpha)(\alpha-t_{J-1})}{q-t_{J-1}}\right]\\
&=\frac{a \eps
q}{\zeta^2}\left[t_J-\alpha+\frac{\zeta\sqrt{\eps}(\alpha-t_{J-1})}{q-t_{J-1}}\right].
\end{split}
\end{equation*}
Moreover,
\[
\psi'(\alpha)=\frac{a\eps q}{(q-\alpha)^2} \quad \mbox{and }
H=x_{J+1}-x_J=\frac{\psi'(\alpha)}{N},
\]
implying that
\[
H=\frac{a \eps q}{N(q-\alpha)^2}.
\]
Therefore,
\[
\begin{split}
H-h_J&=\frac{a\eps q}{q-\alpha}\left[\frac{1}{N(q-\alpha)}-
\frac{t_J-\alpha}{q-\alpha}-\frac{\alpha-t_{J-1}}{q-t_{J-1}}\right]\\
&=\frac{a\eps q}{q-\alpha}\left[
\frac{\alpha-t_{J-1}}{q-\alpha}-\frac{\alpha-t_{J-1}}{q-t_{J-1}}\right]\\
&=\frac{a\eps q}{q-\alpha}\left(\alpha-t_{J-1}\right)\left[
\frac{1}{q-\alpha}-\frac{1}{q-t_{J-1}}\right]\\
&=\frac{a\eps q}{q-\alpha}\left(\alpha-t_{J-1}\right)\frac{\alpha-t_{J-1}}{(q-\alpha)(q-t_{J-1})}\\
&=\frac{a\eps q}{(q-\alpha)^2}\cdot\frac{(\alpha-t_{J-1})^2}{q-t_{J-1}}.\\
\end{split}
\]
We now have
\[
\begin{split}
\eps\frac{H-h_J}{h_JH} &=\frac{a\eps^2
q}{(q-\alpha)^2}\cdot\frac{(\alpha-t_{J-1})^2}{q-t_{J-1}}
\cdot\frac{q-\alpha}{a\eps q}\cdot\frac{1}{\frac{t_J-\alpha}{q-\alpha}+\frac{\alpha-t_{J-1}}{q-t_{J-1}}} \cdot\frac{(q-\alpha)^2N}{a\eps q}\\
&=\frac{(q-\alpha)N}{aq}\cdot\frac{(\alpha-t_{J-1})^2}{q-t_{J-1}}
\cdot\frac{(q-\alpha)(q-t_{J-1})}{\frac{q}{N}-\alpha^2+2\alpha t_{J-1}-t_{J-1}t_J}\\
&=\frac{(q-\alpha)^2N}{aq}\cdot\frac{(\alpha-t_{J-1})^2}{\omega}\le
\frac{\zeta^2\eps}{aqN}\cdot\frac{1}{\omega},
\end{split}
\]
where
\[
\omega:= \frac{q}{N}-\alpha^2+2\alpha t_{J-1} - t_{J-1}t_J
\]
and where in the last step we used~\eqref{eq:zeta} and the fact that
$0\le \alpha-t_{J-1} \le 1/N$. The denominator $\omega$ can be
estimated as follows:
\[
\begin{split}
\omega
&=\frac{q}{N}-\left(\alpha-t_{J-1}\right)^2 - \frac{t_{J-1}}{N}\\
&=\frac{\zeta\sqrt{\eps}+\alpha}{N}-\left(\alpha-t_{J-1}\right)^2 - \frac{t_{J-1}}{N}\\
&=\frac{\zeta\sqrt{\eps}}{N}+\frac{1}{N}\left(\alpha-t_{J-1} \right)-\left(\alpha-t_{J-1}\right)^2\\
&=\frac{\zeta\sqrt{\eps}}{N}+\left(\alpha-t_{J-1} \right)\left(t_J-\alpha\right)\\
&\ge\frac{\zeta\sqrt{\eps}}{N},
\ \ \mbox{since } \left(\alpha-t_{J-1} \right)\left(t_J-\alpha\right)\ge0.\\
\end{split}
\]
Therefore,
\[
\eps\frac{H-h_J}{h_JH} \le
\frac{\zeta^2\eps}{aqN}\cdot\frac{N}{\zeta\sqrt{\eps}}
=\frac{\zeta\sqrt{\eps}}{aq}.
\]
This completes the proof of \eqref{eq:ratio_J}.
\end{proof}

It is easy to see that $\tilde A_N$ is an $L$-matrix. The next lemma
shows that $\tilde A_N$ is an $M$-matrix and that the modified
discretization \eqref{eq:mod_system} is stable uniformly in $\eps$.

\begin{lemma}\label{lem:stability}
Let $\eps$ be sufficiently small, independently of $N$, and let
$a>4/\beta$. Then the matrix $\tilde A_N$ of the system
\eqref{eq:mod_system} satisfies
\[
\left\|\tilde A_N^{-1}\right\|\le C.
\]
\end{lemma}

\begin{proof}
We want to construct a vector $v=[v_0,v_1,\ldots, v_N]^T$ such that
\begin{itemize}
\item[(a)] $v_i\ge \delta$, $i=0,1,\ldots, N$, where $\delta$ is a positive constant independent of both $\eps$ and $N$,
\item[(b)] $v_i\le C$, $i=0,1,\ldots, N$,
\item[(c)] $\sigma_i:=l_i v_{i-1} + d_i v_i + r_i v_{i+1} \ge \delta$, $i=1,2,\ldots, N-1$.
\end{itemize}
Then, according to the $M$-criterion,
\[
\|\tilde A_N^{-1}\| \le \delta^{-1}\|v\| \le C.
\]

The following choice of the vector $v$ is motivated
by~\cite{Roos96,VN14,NV15}:
\[
v_i=
\begin{cases}
\alpha-Hi +\lambda, & i\le J-1,\\
\alpha-Hi+\dfrac{\lambda}{1+\rho_J}(1+\rho)^{J-i}, &i\ge J,\\
\end{cases}
\]
where $\rho_J=\beta h_J/(2\eps)$, $\rho=\beta H/(2\eps)$, and
$\alpha$ and $\lambda$ are fixed positive constants. Since $HN\le
C$, there exists a constant $\alpha$ such that $v_i\ge \alpha-Hi\ge
\delta>0$, so the condition (a) is satisfied. Then, because of
$v_i\le \alpha+\lambda$, the condition (b) holds true if we show
that $\lambda\le C$. We do this next as we verify the condition (c).

When $1\le i\le J-2$, we use \eqref{eq:ratio_less_J-2} to get
\[
\begin{split}
\sigma_i 
&=(l_i+d_i+r_i)v_i+l_iH-r_iH\\
&=\frac{\hbar_i}{H}c_iv_i-\frac{\eps}{h_i}+\frac{\eps}{h_{i+1}}+\frac{b_i\hbar_i}{h_{i+1}}\\
&\ge-\left(\frac{\eps}{h_i}-\frac{\eps}{h_{i+1}}\right)+\frac{b_i}{2}+\frac{b_ih_i}{2h_{i+1}}\\
&=-\frac{\eps(h_{i+1}-h_i)}{h_ih_{i+1}}+\frac{b_i}{2}+\frac{b_ih_i}{2h_{i+1}}\\
&\ge -\frac{2}{a} +\frac{b_i}{2} \ge
\frac{\beta}{2}-\frac{2}{a}=:\delta>0.
\end{split}
\]
The constant $\delta$ exists because of the assumption $a>4/\beta$.

For $i=J-1$, we have
\[
\begin{split}
\sigma_{J-1} 
&=\frac{\hbar_{J-1}}{H}c_{J-1}v_{J-1}+l_{J-1}H-r_{J-1}H\\
&\quad +\lambda l_{J-1}+\lambda d_{J-1}+r_{J-1}\frac{\lambda}{1+\rho_J}\\
&\ge-\frac{\eps}{h_{J-1}}+\frac{\eps}{h_J}+\frac{b_{J-1}\hbar_{J-1}}{h_J}-r_{J-1}\frac{\lambda\rho_J}{1+\rho_J}\\
&\ge -\frac{\eps}{h_{J-1}}+\frac{b_{J-1}}{2}-r_{J-1}\frac{\lambda\rho_J}{1+\rho_J}\\
&\ge -\frac{\eps}{h_{J-1}}+\frac{\beta}{2} + \left(\frac{\eps}{h_JH}+\frac{b_{J-1}\hbar_{J-1}}{h_JH}\right)\frac{\lambda \beta h_J}{2\eps+\beta h_J}\\
&= -\frac{\eps}{h_{J-1}}+\frac{\beta}{2} + \left(\frac{2\eps + b_{J-1}(h_{J-1}+h_J)}{2h_JH}\right)\frac{\lambda \beta h_J}{2\eps+\beta h_J}\\
&\ge \frac{\beta}{2} -\frac{\eps}{h_{J-1}} + \frac{\lambda\beta}{4H}
\ge \frac{\beta}{2} >\delta
\end{split}
\]
with a suitable positive constant $\lambda$. We can choose such
$\lambda$ because the estimates $H\le 2N^{-1}$ and $q-t_{J-1}\le
q-t_{J-2}\le 1$ imply
\[
\frac{\lambda\beta}{4H}-\frac{\eps}{h_{J-1}}=\frac{\lambda\beta}{4H}-\frac{N}{aq}
\left(q-t_{J-1}\right)\left(q-t_{J-2}\right) \ge
N\left(\frac{\lambda\beta}{8}-\frac{1}{aq}\right)\ge 0.
\]

For $i=J$, we get
\begin{small}
\[
\begin{split}
\sigma_J 
&=\frac{\hbar_{J}}{H}c_{J}v_{J}+l_JH-r_JH+\lambda\left[l_J+ \frac{d_J}{1+\rho_J} + \frac{r_J}{(1+\rho_J)(1+\rho)}\right] \\
&\ge-\frac{\eps}{h_J}+\frac{\eps}{H}+\frac{b_J\hbar_J}{H}\\
&\quad +\frac{\lambda}{(1+\rho_J)(1+\rho)}\left[l_J(1+\rho_J)(1+\rho)+ d_J(1+\rho) + r_J\right] \\
&\ge\frac{\eps}{H}-\frac{\eps}{h_J}+\frac{b_J}{2}\\
&\quad +\frac{\lambda}{(1+\rho_J)(1+\rho)}\left[l_J(1+\rho_J)(1+\rho)+ d_J(1+\rho) + r_J\right]\\
&\ge \frac{\beta}{2} -\frac{\eps(H-h_J)}{h_JH}\ge\delta>0.
\end{split}
\]
\end{small}
The above estimate holds true because \eqref{eq:ratio_J} implies
that
\[
\frac{\eps(H-h_J)}{h_JH}\le \frac{\zeta\sqrt{\eps}}{aq}\le
\frac{2}{a},
\]
when $\eps$ is sufficiently small, and because we can show that
\[
\left[l_J(1+\rho_J)(1+\rho)+ d_J(1+\rho) + r_J\right]\ge 0.
\]
Indeed,
\[
\begin{split}
l_J(1+\rho_J)(1+\rho)+ d_J(1+\rho) +
r_J&=l_J\rho_J+l_J\rho_J\rho-r_J\rho\\
&= -\frac{\eps}{h_JH}\frac{\beta h_J}{2\eps}-\frac{\eps}{h_J
H}\frac{\beta h_J}{2\eps}\frac{\beta H}{2\eps}\\
&\quad +\left[\frac{\eps}{H^2}+ \frac{b_J\hbar_J}{H^2}\right]\frac{\beta H}{2\eps}\\
&= -\frac{\beta^2}{4\eps}+\frac{\beta b_J\hbar_J}{2H\eps}\\
&=-\frac{\beta^2}{4\eps}+\frac{\beta b_J(h_J+H)}{4H\eps}\\
&\ge-\frac{\beta^2}{4\eps}+\frac{\beta b_J}{4\eps}\ge 0.
\end{split}
\]

Finally, when $J+1\le i\le N-1$, we have
\[
\begin{split}
\sigma_i 
&=c_iv_i+ l_iH  - r_iH + \frac{l_i}{1+\rho_J} \left[\frac{\lambda}{(1+\rho)^{i-1-J}} - \frac{\lambda}{(1+\rho)^{i-J}}\right] \\
&\quad + \frac{r_i}{1+\rho_J} \left[\frac{\lambda}{(1+\rho)^{i+1-J}} - \frac{\lambda}{(1+\rho)^{i-J}}\right] \\
&\ge b_i + \frac{\rho(1+\rho)l_i-\rho r_i}{(1+\rho_J)(1+\rho)^{i+1-J}}\lambda \\
&\ge \frac{\beta}{2} + \frac{(l_i-r_i + l_i\rho)\rho}{(1+\rho_J)(1+\rho)^{i+1-J}}\lambda \\
&= \frac{\beta}{2} + \left(\frac{b_i}{H}-\frac{\beta}{2H}\right)\frac{\lambda\rho(1+\rho)^{J-i-1}}{1+\rho_J}\\
&\ge \frac{\beta}{2} > \delta.
\end{split}
\]
\end{proof}

By examining the elements of the matrix $\tilde A_N$, we see that
\[
\| \tilde A_N \| \le C N^2.
\]
When we combined this with Lemma \ref{lem:stability}, we get the
following result.

\begin{thm}\label{conditioning}
The matrix $\tilde A_N$ of the system \eqref{eq:mod_system}
satisfies
\[
\kappa(\tilde A_N) \le C N^2.
\]
\end{thm}

\section{Uniform convergence}
Let $\tau_i$, $i=1,2,\ldots,N-1$, be the consistency error of the
finite-difference operator $\mathcal{L}^N$,
\[
\tau_i = \mathcal{L}^Nu_i - f_i.
\]
We have
\[
\tau_i = \tau_i[u] := \mathcal{L}^Nu_i - (\mathcal{L}u)_i
\]
and by Taylor's expansion we get that
\begin{equation}\label{eq:Taylor}
|\tau_i[u]| \le C h_{i+1} (\eps\|u'''\|_i + \|u''\|_i),
\end{equation}
where $\|g\|_i:=\max _{x_{i-1}\le x\le x_{i+1}}^{}|g(x)|$ for any
$C(I)$-function $g$. Let us define
\[
\tilde\tau_i[u] = \left\{ \begin{array}{ll}
\dss\frac{\hbar_i}{H}\tau_i[u], & 1\le i \le J,\\
& \\
\tau_i[u], & J+1\le i \le N-1.
\end{array}\right.
\]

\begin{lemma}\label{lem:consistency_error}
The following estimate holds true for all $i=1,2,\ldots,N-1$:
\[
|\tilde \tau_i[u]|\le CN^{-1}.
\]
\end{lemma}

\begin{proof}
We use the decomposition \eqref{eq:decomp} and estimates
\eqref{eq:decomp_est}. For the smooth part of the solution, it is
easy to show that $|\tilde \tau[s]|\le CN^{-1}$. Then we need to
show that
\[
|\tilde \tau_i[y]| \le CN^{-1}.
\]

\textbf{Case 1.} Let $i\ge J+1$, i.e.~$t_{i-1}\ge t_J\ge\alpha$.
Then we have
\[
\begin{split}
|\tilde \tau_i[y]|=|\tau_i[y]|&\le Ch_{i+1}\left(\eps\|y'''\|_i+\|y''\|_i\right)\\
&\le CN^{-1}\lambda'(t_{i+1})\eps^{-2}e^{-\beta\lambda(t_{i-1})/\eps}\\
&\le CN^{-1}\lambda'(t_{i+1})\eps^{-2}e^{-\beta\lambda(\alpha)/\eps}\\
&\le CN^{-1}\eps^{-2}e^{-a\beta\alpha/(\zeta\sqrt{\eps})}\\
&\le CN^{-1},
\end{split}
\]
where we have used the fact that
$\eps^{-2}e^{-a\beta\alpha/(\zeta\sqrt{\eps})}\le C$.

\textbf{Case 2.} Let $i\le J$, i.e.~$t_{i-1}<\alpha$, and at the
same time, let $t_{i-1}\le q-3/N$. Note that, when $t_{i-1}\le
q-3/N$, we have
\[
t_{i+1}\le q-1/N<q \quad \mbox{and}\quad q-t_{i+1}\ge
\dfrac{1}{3}(q-t_{i-1}).
\]
This is because
\[
q-t_{i-1}\ge \frac{3}{N} \ \ \Rightarrow \ \
\frac{2}{3}(q-t_{i-1})\ge \frac{2}{N},
\]
which gives
\[
q-t_{i+1}=q-t_{i-1}-\frac{2}{N}=\frac{1}{3}(q-t_{i-1})+\frac{2}{3}(q-t_{i-1})-\frac{2}{N}\ge
\frac{1}{3}(q-t_{i-1}).
\]

Therefore,
\[
\begin{split}
|\tilde \tau_i[y]|=\frac{\hbar_i}{H}|\tau_i[y]|&\le \frac{\hbar_i}{H}Ch_{i+1}\left(\eps\|y'''\|_i+\|y''\|_i\right)\\
&\le CN^{-1}\left[\lambda'(t_{i+1})\right]^2\eps^{-2}e^{-\beta\lambda(t_{i-1})/\eps}\\
&\le C N^{-1}\left[\phi'(t_{i+1})\right]^2e^{-a\beta\phi(t_{i-1})}\\
&\le C\eps^{-1}N^{-1}(q-t_{i+1})^{-4}e^{-a\beta(q/(q-t_{i-1})-1)}\\
&\le CN^{-1}(q-t_{i-1})^{-4}e^{-a\beta q/(q-t_{i-1})}\\
&\le CN^{-1},\\
\end{split}
\]
because $(q-t_{i-1})^{-4}e^{-a\beta q/(q-t_{i-1})}\le C$.

\textbf{Case 3.} In the last case, we consider the remaining
possibility, $q-3/N< t_{i-1} < \alpha$. We use the fact that
$\mathcal{L}y=0$ to work with
\[
|\tilde\tau_i[y]| = \frac{h_i}{H}|\tau_i[y]| \le
\frac{\hbar_i}{H}\left(P_i + Q_i+ R_i\right),
\]
where
\[
P_i= \eps |D'' y_i|, \ \  \ \ Q_i = b_i|D'y_i|,  \ \ \mbox{ and} \ \
R_i=c_i|y_i|.
\]
We now follow closely the technique in~\cite[Lemma 5]{Vul01}, (see
also~\cite{VN14,NV15}), to get
\[
\begin{split}
\frac{\hbar_i}{H}\left(P_{i}+Q_i+R_i\right)&\le
C\left[\frac{\hbar_i}{H}\left(\frac{1}{\hbar_i}\eps\cdot 2 \|
y'\|_i\right) +
\frac{\hbar_i}{H}\left(\frac{1}{h_{i+1}}\| y\|_i\right)+ e^{-\beta\lambda(t_{i})/\eps}\right]\\
&\le C N e^{-\beta\lambda(t_{i-1})/\eps}\\
&\le C Ne^{-a\beta\phi(t_{i-1})}\\
&\le C Ne^{-a\beta\phi(q-3/N)}\\
&\le  C Ne^{-a\beta(qN/3-1)}\\
&\le CN^{-1}.
\end{split}
\]
\end{proof}
\begin{remark}
The technique used in the above proof is based on~\cite{Vul83},
where the same approach is successfully applied to
reaction-diffusion problems. This approach is originally due to
Bakhvalov~\cite{Bakh}. The technique works here for
convection-diffusion problems~\eqref{eq:1DCD} because an extra
$\eps$-factor is obtained from the
preconditioner~\eqref{eq:preconditioner}.
\end{remark}

When Lemmas \ref{lem:stability} and \ref{lem:consistency_error} are
combined, which amounts to the use of the consistency-stability
principle, we obtain the following result.

\begin{thm}
Let $\eps$ be sufficiently small, independently of $N$, and let
$a>4/\beta$. Then the solution $U^N$ of the discrete problem
\eqref{eq:matrix_form} on the VB-mesh satisfies
\[
\left\|U^N-u^N\right\|\le CN^{-1},
\]
where $u$ is the solution of the continuous problem \eqref{eq:1DCD}.
\end{thm}


\end{document}